\documentclass[11pt,twoside,openany,reqno]{amsart}

\usepackage{mathrsfs}
\usepackage{latexsym}
\usepackage{amsmath}
\usepackage{amssymb}
\usepackage{amsthm}
\usepackage{amscd}
\usepackage{amsfonts}
\usepackage{epsfig}
\usepackage{multicol}
\usepackage{CJK}
\usepackage{url}
\usepackage{bbm}
\usepackage{amsmath}\numberwithin{equation}{section}

\def\C{\mathbb{C}}

\def\gs{\begin{aligned}}
\def\egs{\end{aligned}}

\theoremstyle{plain}
\newtheorem{thm}{Theorem}[section]

\newtheorem{theorem}[thm]{Theorem}
\newtheorem{lemma}[thm]{Lemma}
\newtheorem{conjecture}[thm]{Conjecture}

\newcommand{\pf}{\noindent\begin {proof}}
\newcommand{\epf}{\end{proof}}

\begin{document}

\title[The  eigenvectors-eigenvalues identity and  Sun's conjectures]
{
The  eigenvectors-eigenvalues identity and  Sun's conjectures on determinants and permanents }
\author{Xuejun Guo}
\address{
Xuejun Guo  \\
Department of Mathematics\\
Nanjing University\\
Nanjing 210093, China}
\email{guoxj@nju.edu.cn}

\author{Xin Li}
\address{
Xin Li \\
Department of Mathematics\\
Nanjing University\\
Nanjing 210093, China}
\email{lix@smail.nju.edu.cn}

\author{Zhengyu Tao}
\address{
Zhengyu Tao \\
Department of Mathematics\\
Nanjing University\\
Nanjing 210093, China}
\email{taozhy@smail.nju.edu.cn}

\author{Tao Wei}
\address{
Tao Wei \\
Department of Mathematics\\
Nanjing University\\
Nanjing 210093, China}
\email{weitao@smail.nju.edu.cn}

\thanks{The authors are supported by National Nature Science Foundation of China (Nos. 11971226, 11631009).}

\date{}

\maketitle

\noindent

\begin{abstract}
In this paper, we prove a conjecture raised by Zhi-Wei Sun in 2018 by the  eigenvectors-eigenvalues identity  
found by  Denton, Parke,  Tao and X.  Zhang in 2019.
\end{abstract}

\medskip



 \section{\bf \Large Introduction}
 
In 2019,  Peter B. Denton, Stephen J. Parke, Terence Tao and Xining Zhang \cite{Tao} found the  legendary eigenvector-eigenvalue identity and showed how this identity can also be used to extract the relative phases between the components of any given eigenvector.    In this paper, we will prove several conjectural identities of Zhi-Wei Sun on  determinants and permanents of matrices involving roots of unity. We will extract the eigenvalues of  the minors $M_j (1\leq j \leq n)$  of  the Hermitian matrix  $A$ from the eigenvectors of $A$ by virtue of the  eigenvector-eigenvalue identity.  So it can been seen as extracting eigenvalues from eigenvectors.

 Zhi-Wei Sun began  a systematical research on permanents and determinants of matrices in number theory  in 2018.  He found 
 many identities with rich arithmetic meanings and also raised many open problems.   His results are in  \cite{sun4},   \cite{sun3},   
\cite{sun2},   \cite{sun1}.  One of Sun's conjectures is the following one which was firstly raised in 2018. 
 \begin{conjecture}[Zhi-Wei Sun,  Conjecture 4.3 of \cite{sun1}]\label{sun}
Let $n>1$ be a integer and $\zeta$ a primitive $n$-th root of unity. 
\begin{enumerate}
	\item If $n$ is a even number, then
	\begin{equation}\label{eq1.1}
		\sum_{\tau \in D(n)} \prod_{j=1}^{n} \dfrac{1}{1-\zeta^{j-\tau(j)}}=\dfrac{\left( \left( n-1\right) !!\right) ^{2}}{2^n}=\dfrac{n!}{4^n}\binom{n}{n/2},
	\end{equation}
\item if $n$ is an odd number, then
\begin{equation}\label{eq1.2}
	\sum_{\tau \in D(n-1)} \prod_{j=1}^{n-1} \dfrac{1}{1-\zeta^{j-\tau(j)}}=\dfrac{1}{n}\left(\dfrac{n-1}{2}! \right)^{2}, 
\end{equation}
and
	\begin{equation}\label{eq1.3}
\sum_{\tau \in D(n-1)} \operatorname{sign}(\tau) \prod_{j=1}^{n-1} \dfrac{1}{1-\zeta^{j-\tau(j)}}=\frac{(-1)^{\frac{n-1}{2}}}{n}\left(\frac{n-1}{2} !\right)^{2} , 
\end{equation}
where $D(n-1)$ is the set of all derangements $\tau$ of indices $j=1, \ldots, n-1$ such that $\tau(j) \neq j$ for all $j=1, \ldots, n-1$.
\end{enumerate}
\end{conjecture}

\section{\bf \Large Eigenvalues from eigenvectors}
In this section, we will prove the equation \eqref{eq1.3}. 
The left hand side of \eqref{eq1.3} is the determinant of  an $(n-1)\times (n-1)$ sub-matrix of the $n\times n$ matrix  $A=(a_{jk})$ such that 
$
a_{j k}=\left(1-\delta_{j k}\right)\left\{1+i \cot \frac{(j-k) \pi}{n}\right\}
$, where $\delta_{j k}=1$ if $j=k$; and $0$ if $j\neq k$.  The following theorem gives the eigenvalues and eigenvectors of $A$. 
\begin{theorem}[F. Calogero and A.M. Perelomov, Theorem 1 of \cite{CP}]\label{CP}
The off-diagonal hermitian matrix of order $n$ whose elements are defined by the formula
$$
a_{j k}=\left(1-\delta_{j k}\right)\left\{1+i \cot \frac{(j-k) \pi}{n}\right\}
$$
has the integer eigenvalues
$$
\lambda_{i}=2 i-n-1, \quad i=1,2, \ldots, n
$$
and the corresponding eigenvectors $v_{i}$ with norm $1$ have components
$$
v_{i,j}=\exp \left(-\frac{2 \pi {\rm{i}} i j}{n}\right)/\sqrt{n}, \quad j=1,2, \ldots, n .
$$\end{theorem}

If $A$ is an $n \times n$ Hermitian matrix, then we denote its $n$ real eigenvalues by $\lambda_{1}(A), \ldots, \lambda_{n}(A)$. We can find an  orthonormal basis  $v_1, v_2, \cdots, v_n\in \C^n$ such that $v_i$ is associated to $\lambda_{i}(A), 1\leq i\leq n$.  For any $i, j=1, \ldots, n$, let  $ v_{i, j} $ denote the $ j^{\text {th }}$ component of $ v_{i}.$ 
 If $1 \leq j \leq n$, let $M_{j}$ denote the $(n-1) \times (n-1)$ minor formed from $A$ by deleting the $j^{\text {th }}$ row and column from $A$. This is also a Hermitian matrix, and thus has $n-1$ real eigenvalues $\lambda_{1}\left(M_{j}\right), \ldots, \lambda_{n-1}\left(M_{j}\right)$.

\begin{theorem}[Eigenvector-eigenvalue identity, Theorem 1 of \cite{Tao}]\label{Tao}
With the notation as above, we have
\begin{equation}\label{tao2}
\left|v_{i, j}\right|^{2} \prod_{k=1 ; k \neq i}^{n}\left(\lambda_{i}(A)-\lambda_{k}(A)\right)=\prod_{k=1}^{n-1}\left(\lambda_{i}(A)-\lambda_{k}\left(M_{j}\right)\right) .
\end{equation}\end{theorem}

{\it \bf Proof of Equation \eqref{eq1.3}:}

Note that  $$ \begin{aligned}\frac{1}{1-\zeta^{j-k}}
&=\frac{1}{2}\left(1+\frac{1+\zeta^{j-k}}{1-\zeta^{j-k}}\right)\\
&=\frac{1}{2}\left(1+{\rm{i}}\cot\frac{\pi (j-k)}{n}\right).\end{aligned}
$$
Let $A=(A_{jk})$ be the $n \times n$ Hermitian matrix  defined  in Theorem \ref{CP}.  By Theorem \ref{CP},  the eigenvalues of  $A$ are \begin{equation}\label{eg2}
\lambda_1=n-1, \lambda_2=n-3, \cdots, \lambda_{i}=n+1-2i,  \cdots, \lambda_n= 1-n.
\end{equation}
Let $i=\frac{n+1}{2}$ and $j=n$. Then $\lambda_i=0$. By Theorem \ref{Tao}, 
 \[
 \left|v_{i, n}\right|^{2} \prod_{k=1 ; k \neq i}^{n}\left(-\lambda_{k}(A)\right)=\prod_{k=1}^{n-1}\left(-\lambda_{k}\left(M_{n}\right)\right) .
 \] 
By Theorem \ref{CP}, $\left|v_{s, n}\right|^{2}=\frac{1}{n}$. Hence 
 $$|M_n|={\prod_{k=1}^{n-1}\lambda_{k}\left(M_{n}\right) }=\frac{(-1)^{\frac{n-1}{2}}((n-1)!!)^2}{n}.$$
 Hence 
$$ \begin{aligned}
\sum_{\tau \in D(n-1)} \operatorname{sign}(\tau) \prod_{j=1}^{n-1} \frac{1}{1-\zeta^{j-\tau(j)}}&=2^{1-n}|M_n|=\frac{2^{1-n}(-1)^{\frac{n-1}{2}}((n-1)!!)^2}{n}\\
&=\frac{(-1)^{\frac{n-1}{2}}}{n}\left(\frac{n-1}{2} !\right)^{2} .
\end{aligned}
$$
Hence Equation \eqref{eq1.3} is proved.
\qed

Han Wang and Zhi-Wei Sun  prove \cite{sun0}  another  conjecture involving derangements and roots of unity by similar arguements.  
My colleague Keqin Liu suggested a different strategy to prove Conjecture \ref{sun}.  Although the eigenvalues of $A$ are rational integers as in Theorem \ref{CP},   the Numerical computation shows that 
there is no simple pattern of the eigenvalues of $M_n$ in the proof of equation \eqref{eq1.3} . However the   eigenvalues of $M_nB$ are very simple for carefully chosen
diagonal matrices. 
Let  $B_{s}$  the  diagonal matrix whose diagonal entries are given as
$
1-\zeta^{i s}, \quad 1 \leq i \leq n-1,
$
where $s \in\left\{-\frac{n-1}{2},-\frac{n-3}{2}, \cdots,-1,1,2, \cdots, \frac{n-1}{2}\right\}$.  Let  $M_{n,s} = M_n B_{s}$.  Then 
 \eqref{eq1.3} is equivalent to 
 \begin{equation}\label{liu}
 \operatorname{det}\left(M_{n,1}\right)=(-1)^{\frac{n-1}{2}}\left(\frac{n-1}{2} !\right)^{2}.
 \end{equation}
Liu found that  the eigenvalues of $M_{n,1}$ for are  $$-\frac{n-1}{2},  \cdots, -1, 1, \cdots, \frac{n-1}{2}$$ by numerical computation. He discusses  this strategy in \cite{Liu}.  

Let $f(x)=\prod_{k=1}^{n-1}\left(x-\lambda_{k}\left(M_{j}\right)\right)$. Then by \eqref{tao2} and \eqref{eg2}, 
 \begin{equation}\label{eq2.4}f(\lambda_i(A))=f(n+1-2i)=\frac{1}{2n} \prod_{k=1 ; k \neq i}^{n}\left(k-i\right), \ 1\leq i\leq n.  \end{equation}
 Hence $f(x)$ is completely determined by the Lagrange interpolation.
\section{\bf The permanents of matrices}

In this section, we will prove equation \eqref{eq1.1} and equation \eqref{eq1.2}. 

\begin{lemma}\label{lam3.2}
	Let $n$ be an integer greater than $2$, and $x_{1},\cdots,x_{n}$  pairwise distinct complex numbers. 
	Let  $L(n)$ be the set of all elements  $\tau\in S(n)$ such that $\tau$ is an  $n$-cycle. Then  \[\sum_{\tau\in L(n)}\prod_{j=1}^{n}\frac{1}{x_{\tau(j)}-x_j}=0.\]
\end{lemma}
\begin{proof}
	For $\tau_{1}=(1\ a_{2}\ a_{3}\ \cdots\ a_{n}), \tau_{2}=(1\ b_{2}\ b_{3}\ \cdots\ b_{n})\in L(n)$, define $\tau_{1}\sim\tau_{2}$ if and only if $(a_{2}\ a_{3}\ \cdots\ a_{n})=( b_{2}\ b_{3}\ \cdots\ b_{n})\in L(n-1)$. One can easily see  that  $\sim$ is an equivalence relation on $L(n)$. Obviously, there are $n-1$ elements in each equivalence class.
	Take an arbitrary equivalence class
	\begin{align*}
		\sigma_1& = (1\ a_2\ a_3\ \cdots\ a_{n-1}\ a_{n}),\\
		\sigma_2& = (1\ a_3\ a_4\ \cdots\ a_{n}\ a_{2}),\\
		&\ \ \ \ \ \vdots\\
		\sigma_{n-1}& = (1\ a_n\ a_2\ \cdots\ a_{n-2}\ a_{n-1}).
	\end{align*}
Then we just need to show $\displaystyle \sum_{i=1}^{n-1}\prod_{j=1}^{n}\frac{1}{x_{\sigma_i(j)}-x_j}=0$. Let $\displaystyle \lambda_i=\prod_{j=1}^{n}\frac{1}{x_{\sigma_i(j)}-x_j}$, obviously, $\alpha_i\neq0$ and \begin{equation*}
	\frac{\lambda_i}{\lambda_1}=\left\lbrace 
	\begin{array}{ll}
		\frac{(x_1-x_{a_2})(x_{1}-x_{a_n})}{x_{a_2}-x_{a_n}}\left(\frac{1}{x_1-x_{a_{i+1}}}-\frac{1}{x_{1}-x_{a_i}}\right),&i\neq1,\\
		1,&i=1.
	\end{array}\right.
\end{equation*}
Thus $\displaystyle \sum_{i=1}^{n-1}\frac{\lambda_i}{\lambda_1}=1+\frac{(x_1-x_{a_2})(x_{1}-x_{a_n})}{x_{a_2}-x_{a_n}}\left(\frac{1}{x_1-x_{a_n}}-\frac{1}{x_1-x_{a_2}}\right)=0$. Hence we have  $\displaystyle \sum_{i=1}^{n-1}\alpha_i=0.$ This complete the proof of lemma \ref{lam3.2}.
\end{proof}

\begin{theorem}\label{thm3.1}
	Let  $n$ be a positive integer, and $x$ a complex number  such that  $x^k\neq 1$ for any  positive integer $k<n$. Assume \[A=(a_{ij})_{n\times n}=\displaystyle \left((1-\delta_{ij})\left(1-\frac{1}{1-x^{i-j}}\right)\right)_{n\times n},\]and \[B=(b_{ij})\in M_{(n-k)\times(n-k)}(\mathbb{C})\] be the sub-matrix obtained from $A$ by deleting the $s_1th,\cdots,s_kth$ columns and the
	$s_1th,\cdots,s_k th$ rows, where  $k=0$ or $2\leqslant k<n,\ 1\leqslant s_{1}<s_{2}<\cdots <s_k\leqslant n$.\begin{enumerate}
		\item If $n-k$ is an odd number, then\[\sum_{\substack{\tau\in D(n-k)\\   \mathrm{sign}(\tau)=1}}\prod_{j=1}^{n-k}b_{j\tau(j)}=\sum_{\substack{\tau\in D(n-k)\\   \mathrm{sign}(\tau)=-1}}\prod_{j=1}^{n-k}b_{j\tau(j)}=0,\]
		\item If $n-k$ is a even number, then \[\sum_{\substack{\tau\in D(n-k)\\   \mathrm{sign}(\tau)=(-1)^{\frac{n-k}{2}+1}}}\prod_{j=1}^{n-k}b_{j\tau(j)}=0.\]
	\end{enumerate}
\end{theorem}

\pf
First, let's check this theorem when $\ell=n-k=1,2,3.$ If $\ell=1$, then $B=0$ and (1) holds; if $\ell=2$, then $B$ looks like
\[\begin{pmatrix}
	0 & \frac{1}{1-x^{-1}}\\
	\frac{1}{1-x} & 0 
\end{pmatrix},x\neq0.\]
Since there is only a odd permutation in $D(2)$, thus (2) holds; if $\ell=3,$ then $B$ looks like \[xy\begin{pmatrix}
0 & \frac{1}{x-1} & \frac{1}{y-1}\\
\frac{1}{1-x} & 0 & \frac{1}{y-x} \\
\frac{1}{1-y} & \frac{1}{x-y} & 0 \\ 
\end{pmatrix},x,y\neq 1\ \text{and}\ x\neq y.\]
Since there are only two odd permutations $(1\ 2\ 3)$, $(1\ 3\ 2)$ in $D(3)$ and  \[\frac{1}{x-1}\cdot\frac{1}{y-x}\cdot\frac{1}{1-y}+\frac{1}{y-1}\cdot\frac{1}{1-x}\cdot\frac{1}{x-y}=0.\]
This implies that (1) holds

If $\ell\geqslant4$ is odd, then \[\displaystyle \sum_{\substack{\tau\in D(\ell)\\   \mathrm{sign}(\tau)=1}}\prod_{j=1}^{\ell}b_{j\tau(j)}=x_1\cdots x_\ell\sum_{\substack{\tau\in D(\ell)\\   \mathrm{sign}(\tau)=1}}\prod_{j=1}^{\ell}\frac{1}{x_{\tau(j)}-x_j},\]

where $x_1,\cdots,x_\ell$ are different from each other.
For $Y=\{a_{1},\cdots,a_{r}\}\subset X=\{x_{1},\cdots,x_{\ell}\}$, define \[f(Y)=\sum_{\tau\in L(r)}\prod_{j=1}^{r}\frac{1}{a_{\tau(r)}-a_j}.\]By lemma \ref{lam3.2} we know that $f(X)=0$ when $|X|>2$. Note that
\begin{equation}\label{disjointunion}
	\sum_{\substack{\tau\in D(\ell)\\
			\mathrm{sign}(\tau)=1}}\prod_{j=1}^{\ell}\frac{1}{x_{\tau(j)}-x_j}=\sum_{\substack{X=\bigsqcup_{i=1}^{s}X_i\\|X_i|\geqslant 2\\ s\ \text{odd}}}\left(\prod_{j=1}^sf(X_i)\right).
\end{equation}
	Since $\ell$ is odd, every disjoint union $X=\bigsqcup_{i=1}^{s}X_i$ in the RHS of (\ref{disjointunion}) has an $X_i$ with $|X_i|>2$, thus\[\sum_{\substack{\tau\in D(\ell)\\
		\mathrm{sign}(\tau)=1}}\prod_{j=1}^{\ell}\frac{1}{x_{\tau(j)}-x_j}=0.\]
	Similarly, it can be shown that $\displaystyle \sum_{\substack{\tau\in D(\ell)\\
			\mathrm{sign}(\tau)=-1}}\prod_{j=1}^{\ell}\frac{1}{x_{\tau(j)}-x_j}=0.$
	
	If $\ell$ is even, one can prove (2) in the same way by using lemma \ref{lam3.2}. This complete the proof of this theorem.
	\epf
{\bf Proof of equation \eqref{eq1.1} and equation \eqref{eq1.2}: }	

	In the notations of Theorem  \ref{thm3.1},  let $x=\zeta$. If $k=0$ and $n$ is a even number, then we have \[\sum_{\substack{\tau\in D(n)\\   \mathrm{sign}(\tau)=(-1)^{\frac{n}{2}+1}}}\prod_{j=1}^{n}a_{j\tau(j)}=0,\]which implies \[\sum_{\tau\in D(n)}\prod_{j=1}^{n}a_{j\tau(j)}=(-1)^{\frac{n}{2}}\sum_{\tau\in D(n)}\mathrm{sign}(\tau)\prod_{j=1}^{n}a_{j\tau(j)}=\dfrac{\left( \left( n-1\right) !!\right) ^{2}}{2^n}.\]
If $k=1, s_{k}=1$, and $n$ is an odd number, then we have \[\sum_{\substack{\tau\in D(n-1)\\   \mathrm{sign}(\tau)=(-1)^{\frac{n+1}{2}}}}\prod_{j=1}^{n-1}a_{j\tau(j)}=0,\]which implies that \[\sum_{\tau\in D(n-1)}\prod_{j=1}^{n-1}a_{j\tau(j)}=(-1)^{\frac{n-1}{2}}\sum_{\tau\in D(n)}\mathrm{sign}(\tau)\prod_{j=1}^{n+1}a_{j\tau(j)}=\dfrac{1}{n}\left(\dfrac{n-1}{2}! \right)^2 .\]
Thus we complete the proof of \eqref{eq1.1} and \eqref{eq1.2}.\qed
\vskip 2cm

{\bf Acknowledgements:}   The author is deeply grateful  to his colleagues Zhi-Wei Sun and Keqin Liu for  very  helpful discussions.

\end{document}